\documentclass[11pt,a4paper]{article}
\usepackage[latin1]{inputenc}
\usepackage{amsmath}
\usepackage{amsfonts}
\usepackage{amssymb}
\usepackage{amsthm}
\usepackage{verbatim}
\usepackage{mathtools}
\usepackage{fancyhdr}
\usepackage{enumerate}

\usepackage[demo]{graphicx}
\usepackage{caption}
\usepackage{subcaption}

\usepackage[T1]{fontenc}

\usepackage{graphicx, color}
\usepackage{tikz}
\usetikzlibrary{shapes.geometric}
\usetikzlibrary{calc}
\usetikzlibrary{decorations.pathreplacing}

\usepackage[numbers,sort&compress]{natbib}

\usepackage{tikz}
\usetikzlibrary{calc}

\usepackage{float}

\newtheorem{theorem}{Theorem}
\newtheorem{lemma}[theorem]{Lemma}

\newtheorem*{claim*}{Claim}

\theoremstyle{definition}

\DeclarePairedDelimiter\floor{\lfloor}{\rfloor}

\DeclareMathOperator{\pw}{pw}

\newcommand{\lr}[1]{\left(#1\right)}
\newcommand{\etal}{\textit{et.~al.\@} }
\renewcommand{\leq}{\leqslant}
\renewcommand{\geq}{\geqslant}

\usepackage{graphicx}
\usepackage[left=3cm, right=3cm, top=2.50cm, bottom=2.50cm]{geometry}

\title{\textbf{Anagram-free Graph Colouring}}
\author{Tim E. Wilson\quad David R. Wood\footnote{Research supported by Australian Research Council}\\
\\
School of Mathematical Sciences\\
Monash University\\
Melbourne, Australia\\
\texttt{\{timothy.e.wilson, david.wood\}@monash.edu}}

\date{July 1, 2016}

\begin{document}
	
	\maketitle

\begin{abstract}
	An anagram is a word of the form $WP$ where $W$ is a non-empty word and $P$ is a permutation of $W$. We study anagram-free graph colouring and give bounds on the chromatic number. Alon et al. (2002) asked whether anagram-free chromatic number is bounded by a function of the maximum degree. We answer this question in the negative by constructing graphs with maximum degree 3 and unbounded anagram-free chromatic number. We also prove upper and lower bounds on the anagram-free chromatic number of trees in terms of their radius and pathwidth. Finally, we explore extensions to edge colouring and $k$-anagram-free colouring.

\end{abstract}

\section{Introduction}
	A \textit{square} is a word of the form $WW$ where $W$ is a non-empty word. A word $X$ is \textit{square-free} if no subword of $X$ is a square. Much of the study of squares involves avoiding squares or, equivalently, characterising the set of square-free words. An early result in this area is the construction, by  Thue \cite{thue1914probleme}, of arbitrarily long square-free words on three symbols. More recently, Alon \etal \cite{alon2002nonrepetitive} generalized the concept of square-free words to graph colouring by requiring colourings to avoid square sequences along subpaths of a graph. A \textit{vertex colouring} of a graph $G$ is a function $f:V(G) \to C$ where $C$ is a set of colours and $V(G)$ is the vertex set of $G$. A vertex colouring of a graph $G$ is \textit{square-free} if the sequence of colours on each subpath of $G$ is not a square. The \textit{square-chromatic number} of a graph $G$ is the minimum number of colours in a square-free colouring of $G$. Square-free graph colouring has been extensively studied and is often referred to as nonrepetitive colouring \cite{britten1971repetitive,grytczuk2007nonrepetitive, grytczuk2013new}. The square-free chromatic number of $G$ is denoted $\pi(G)$ and is also known as the \textit{Thue number} or \textit{nonrepetitive chromatic number}. With this notation, Thue's result is $\pi(P) \leq 3$ for every path $P$.
	
	In this paper we introduce anagram-free graph colouring, which is a topic suggested for study by Alon \etal \cite{alon2002nonrepetitive}. This definition follows naturally from the established notion of anagrams in combinatorics on words \cite{cori1990partially}. An \textit{anagram} is a word of the form $W_1W_2$ where $W_1$ is a non-empty word and $W_2$ is a permutation of $W_1$. A word is \textit{anagram-free} if it contains no anagram as a subword. An anagram is also called an \textit{abelian square} and an anagram-free word is also called an \textit{abelian square-free word} or a \textit{strongly non-repetitive sequence}. Our generalization to graph colouring follows the example set by square-free colouring. A vertex colouring of a graph $G$ is \textit{anagram-free} if the sequence of colours on each subpath of $G$ is not an anagram. The \textit{anagram-free chromatic number} of $G$ is the minimum number of colours in an anagram-free colouring of $G$, denoted by $\phi(G)$.
	
	Results analogous to Thue's exist for anagram-free words. Anagram-free words on three symbols have length at most $7$, for example $abcbabc$ \cite{cummings1996strongly}. Ker{\"a}n \cite{keranen1992abelian, keranen2009abelian} constructed anagram-free words of arbitrary length on four symbols. In the context of anagram-free graph colouring, Ker{\"a}nen's result is $\phi(P) \leq 4$ for every path $P$. Throughout this paper we prove bounds on $\phi$ that demonstrate significant differences in the behaviour of $\phi$ and $\pi$. This is somewhat surprising considering the similarity of Ker{\"a}nen's and Thue's results.
	
	\subsection{Bounding by maximum degree}
	
	An area of study central to square-free colouring is bounding $\pi(G)$ by functions of the maximum degree, $\Delta(G)$. Alon \etal \cite{alon2002nonrepetitive} proved a result that implies $\pi(G) \leq c\Delta(G)^2$ for some constant $c$. Several subsequent works improved the value of $c$ \cite{fertin2004star, grytczuk2006nonrepetitive, harant2012nonrepetitive} with the best known value being $c = 1 + o(1)$ \cite{dujmovic2011nonrepetitive}. Earlier proofs use the Lov\'{a}sz Local Lemma while the proof in \cite{dujmovic2011nonrepetitive} uses entropy compression.
	
	Alon \etal \cite{alon2002nonrepetitive} states as an open problem whether $\phi$ is bounded by a function of maximum degree. We answer this question, thus finding the first significant difference between $\pi$ and $\phi$. We prove that, unlike $\pi$, no function of maximum degree is an upper bound on $\phi$.
	\begin{theorem}\label{thm:noDegreeBound}
		Graphs of maximum degree $3$ have unbounded anagram-free chromatic number.
	\end{theorem}
	Theorem \ref{thm:noDegreeBound} complements the result of Richmond and Shallit \cite{richmond2009counting} regarding enumeration of anagrams. They counted anagrams with the aim of using the Lov\'{a}sz Local Lemma to prove that $\phi$ is bounded on paths. They concluded that the probability that a random word is an anagram is too large for this approach. A proof using the Lov\'{a}sz Local Lemma that $\phi$ is bounded on paths would likely apply to graphs of maximum degree $3$, which Theorem \ref{thm:noDegreeBound} shows is impossible.
	
	We also consider variations of square-free colouring and apply them to anagram-free colouring. An \textit{edge colouring} of a graph is an assignment of colours to the edges of the graph. An edge colouring of $G$ is \textit{square-free} if every subpath of $G$ has a square-free colour sequence along its edges. Alon \etal \cite{alon2002nonrepetitive} defined the \textit{square-chromatic index}, denoted $\pi'(G)$, as the minimum number of colours in a square-free edge colouring of $G$. We define anagram-free edge colouring similarly. An edge colouring of a graph $G$ is \textit{anagram-free} if every subpath of $G$ has an anagram-free colour sequence along its edges. The corresponding \textit{anagram-free chromatic index} is denoted $\phi'(G)$.
	
	The obvious bound, $\phi'(G) \geq \Delta(G)$, follows from the observation that edges incident to a common vertex receive distinct colours in an anagram-free edge colouring. We prove a significant improvement on this bound with the following result.
	\begin{theorem}\label{thm:noDegreeBoundEdge}
		Trees of maximum degree $3$ have unbounded anagram-free chromatic index.
	\end{theorem}
	Theorem \ref{thm:noDegreeBoundEdge} is for a more restricted class of graphs than Theorem \ref{thm:noDegreeBound}, so Theorem \ref{thm:noDegreeBoundEdge} says more about $\phi'$ than Theorem \ref{thm:noDegreeBound} does about $\phi$. The degree bound of $3$ in Theorems \ref{thm:noDegreeBound} and \ref{thm:noDegreeBoundEdge} is best possible since $\phi(P) \leq 4$ for all paths $P$, which implies $\phi(C) \leq 5$ for all cycles $C$. Indeed, to anagram-free $5$-colour a cycle assign one vertex a unique colour and then anagram-free $4$-colour the remaining vertices. The proof that $\phi'(C) \leq 5$ is analogous. Whether this bound can be improved to $4$ is an open problem. The analogous problem for square-free colouring was solved by Currie \cite{currie2002there}, who showed that $\pi(C) \leq 3$ for every cycle with the exception of order $5$, $7$, $9$, $10$, $14$ and $17$.
	
	\subsection{Bounding $\phi$ with given radius or pathwidth}
	
	Our initial bounds for $\phi$ and $\phi'$ motivate further study of $\phi$ for trees. First, note that Bre{\v{s}}ar \etal \cite{brevsar2007nonrepetitivetree} proved that $\pi(T) \leq 4$ for every tree $T$. In contrast, we prove the following:
	\begin{theorem}\label{thm:treeDegreeBound}
		Trees have unbounded anagram-free chromatic number.
	\end{theorem}
	We also study $\phi$ for trees as a function of radius. The \textit{radius} of a tree $T$ is the minimum, taken over all vertices $v$ in $T$, of the maximum distance of a vertex from $v$. Such a vertex $v$ is called a \textit{centre} of $T$. We obtain the following tight bound on $\phi$ as a function of radius.
	\begin{theorem}\label{thm:treeHeightBound}
		Every tree $T$ of radius $h$ has $\phi(T) \leq h + 1$. Moreover, for every $h \geq 0$ there is a tree $T$ of radius $h$ such that $\phi(T) \geq h$.
	\end{theorem}
	The bound in Theorem \ref{thm:treeHeightBound} is poor for paths since the radius of a path is roughly half its length. To address this, we prove a bound on $\phi$ for trees that is bounded for paths. To do so we use pathwidth, which is a well studied parameter in square-free colouring \cite{gkagol2016pathwidth, dujmovic2011nonrepetitive} as well as more generally. Let $\pw(T)$ denote the pathwidth of a tree $T$. 
	
	\begin{theorem}\label{thm:treePathwidthBound}
		For every tree $T$, $\phi(T) \leq 4\pw(T) + 1$. Moreover, for every $p \geq 0$ there is a tree $T$ such that $\phi(T) \geq p \geq \pw(T)$.
	\end{theorem}
	
	Note that since every tree $T$ on $n$ vertices has pathwidth $O(\log n)$, Theorem \ref{thm:treePathwidthBound} implies that $\phi(T) \leq O(\log n)$. It is open whether $\phi(G) \leq 5$ for graphs of pathwidth $1$ because not all trees of pathwidth $1$ are paths.
	
	\subsection{$k$-anagram-free colouring}
	
	We also consider the $k$-power generalisation of square-free colouring and apply it to anagram-free colouring. A \textit{$k$-power} is a word $W^k$ where $W$ is a non-empty word, for $k \geq 2$. In this setting a square is a $2$-power. A \textit{$k$-power-free colouring} is a colouring avoiding paths with $k$-power colour sequences. This definition can be applied for vertex or edge colouring. The corresponding \textit{$k$-power-free chromatic number} is $\pi_k(G)$ and the \textit{$k$-power-free chromatic index} is $\pi'_k(G)$ \cite{grytczuk2007nonrepetitive}. 
	
	In this paper, we introduce $k$-anagram-free colouring. $k$-anagrams, often called abelian $k$-powers, are an established object of study in combinatorics on words \cite{dekking1979strongly}. An \textit{$k$-anagram} is a word $W_1W_2\ldots W_k$ where each $W_i$ is a permutation of a non-empty word $W$, for $k \geq 2$. A colouring of a graph $G$ is \textit{$k$-anagram-free} if the sequence of colours on each subpath of $G$ is not a $k$-anagram. We apply this definition to both vertex and edge colourings. The corresponding \textit{$k$-anagram-free chromatic number} is denoted by $\phi_k(G)$ and the \textit{$k$-anagram-free chromatic index} is denoted by $\phi'_k(G)$.
	
	Every $k$-anagram contains an $(k-1)$-anagram so a $k$-anagram-free colouring it is also $(k+1)$-anagram-free. Thus, for every graph $G$,
	\begin{align}\label{eqn:kChain}
	\phi(G) = \phi_2(G) \geq \phi_3(G) \geq \phi_4(G) \geq \cdots
	\end{align}
	with an analogous expression for $\phi'_k$. Therefore we can immediately apply upper bounds for $\phi$ and $\phi'$ to $\phi_k$ and $\phi'_k$ respectively. With this in mind we first study lower bounds of $\phi_k$ and $\phi'_k$ and start by studying bounds as a function of maximum degree. We prove the following generalisation of Theorem \ref{thm:noDegreeBound}.
	\begin{theorem}\label{thm:noKPowerBound}
		For $k \geq 2$, $k$-anagram-free chromatic number is unbounded on graphs of maximum degree $k+1$.
	\end{theorem}
	Theorem \ref{thm:noDegreeBound} is implied by Theorem \ref{thm:noKPowerBound} with $k=2$. Note that the degree bound in Theorem \ref{thm:noKPowerBound} depends on $k$ and it is open whether such a result holds for a degree bound independent of $k$. We also investigate upper bounds for $k$-anagram-free colouring and prove the following contrasting result.
	\begin{theorem}\label{thm:kPowerTreeBound}
		If $k \geq 4$ then $\phi_k(T) \leq 4$ and $\phi'_k(T) \leq 4$ for every tree $T$.
	\end{theorem}
	This result is somewhat surprising given Theorems \ref{thm:noDegreeBound} and \ref{thm:treeDegreeBound} which say that $\phi_2$ and $\phi'_2$ are unbounded on trees. Theorem \ref{thm:kPowerTreeBound} leaves a gap at $k=3$, which motivates the question of whether $\phi_3$ and $\phi'_3$ are bounded on trees. We have upper bounds for $\phi_3$ on trees, in terms of radius and pathwidth, due to Equation (\ref{eqn:kChain}) and Theorems \ref{thm:treeHeightBound} and \ref{thm:treePathwidthBound}. We prove a similar upper bound for $\phi'_3$ in the following theorem.
	\begin{theorem}\label{thm:kTreeIndexPathwidth}
	 	For every tree $T$, $\phi'_3(T) \leq 4\pw(T)$.
	\end{theorem}
	These bounds for $\phi_3$ and $\phi'_3$ depend on pathwidth or radius. The question of whether $\phi_3$ and $\phi'_3$ are bounded on trees is an open problem.
	
	We also give tighter bounds than those in Theorem \ref{thm:kPowerTreeBound} for larger values of $k$. We improve the bound to $3$ for $k \geq 6$ and to $2$ for $k \geq 8$ by using results on $3$-anagram-free and $4$-anagram-free words \cite{dekking1979strongly}.
	
	\section{Lower Bounds}\label{sec:lower}
	In this section we prove lower bounds for $\phi$ and $\phi'$. Most of these bounds depend upon counting the occurrence of colours in each half of a path, which motivates the following definition. A \textit{colour multiset} of size $n$ on $c$ colours is a multiset of size $n$ with entries from $[c]$ where $[c] := \{1,2,\ldots,c\}$. Let $\mathcal{M}_{n,c}$ be the set of colour multisets of size $n$ on $c$ colours. Note that $|\mathcal{M}_{n,c}|$ equals the number of ways to place $n$ unlabelled balls in $c$ labelled boxes, which is well known to equal $\binom{n + c - 1}{ c - 1}$; see \cite[Section 1.9]{stanley2011enumerative}. We give a weaker bound by noting that the number of occurrences of each colour is in $\{0, 1, \ldots ,n\}$, therefore
	\begin{align}\label{eqn:colourMult}
		|\mathcal{M}_{n,c}| \leq (n+1)^c.
	\end{align}
	This simple bound suffices for our needs because we only require that, once $c$ is fixed, $|\mathcal{M}_{n,c}|$ is bounded by a polynomial in $n$. For a coloured graph $G$, let $M(G)$ be the multiset of colours that occur in $G$ and call $M(G)$ the \textit{colour multiset} of $G$. A \textit{$c$-colouring} of a graph $G$ is a colouring of $G$, either for vertices or edges, where the colour set has size $c$. Note that if $G$ is vertex $c$-coloured then $M(G) \in \mathcal{M}_{|V(G)|, c}$, and if $G$ is edge $c$-coloured then $M(G) \in \mathcal{M}_{|E(G)|, c}$. We intentionally allow the definition to apply to edge or vertex colouring and in each case the usage will be clear from the context.
	
	The function $M$ is useful for the analysis of anagram-free colouring. Note that a vertex coloured path $v_1,\ldots,v_{2i}$ is an anagram if and only if
	\begin{align*}
	M(v_1,\ldots,v_i)=M(v_{i+1},\ldots,v_{2i}).
	\end{align*}
	Similarly, for $i \geq 1$, an edge coloured path $v_1,\ldots,v_{2i + 1}$ is an anagram if and only if
	\begin{align*}
	M(v_1v_2,\ldots,v_iv_{i + 1})=M(v_{i+1}v_{i+2},\ldots,v_{2i}v_{2i + 1}).
	\end{align*}
	The indices in these equations highlight a distinction between $\phi$ and $\phi'$. In a vertex colouring only the paths of even order can be anagrams. However, in an edge colouring only the paths of even length can be anagrams.
	
	\subsection{Edge colouring}
	
	We start with the proof of Theorem \ref{thm:noDegreeBoundEdge}, which says that $\phi'$ is unbounded on trees of maximum degree $3$. The proof uses the fact that the number of leaves in a complete binary tree grows exponentially with height while, for fixed $c$, $|\mathcal{M}_{n,c}|$ is bounded by a polynomial in $n$. We then associate colour multisets to leaves and show that the tree contains an anagram if two leaves share a colour multiset. A \textit{complete binary tree} is a rooted tree such that every non-leaf vertex has two children and the leaves have equal distance to the root.
	
	\begingroup
	\def\thetheorem{\ref{thm:noDegreeBoundEdge}}
	\begin{theorem}
	Trees of maximum degree $3$ have unbounded anagram-free chromatic index.
	\end{theorem}
	\addtocounter{theorem}{-1}
	\endgroup 
	\begin{proof}
		Fix $c \geq 1$ and choose $h \in \mathbb{Z}^+$ so that $2^{h} > (h + 1)^c$. Let $T$ be the rooted complete binary tree of height $h$ with root vertex $r$. Fix an arbitrary edge $c$-colouring of $T$. By our choice of $h$ and Equation (\ref{eqn:colourMult})
		\begin{align*}
			\#\text{leaves} = 2^h > (h+1)^c \geq |\mathcal{M}_{h, c}|.
		\end{align*}
		Since each root-to-leaf path in $T$ has $h$ edges, the number of leaves in $T$ is greater than the number of distinct colour multisets on root-to-leaf paths in $T$. Therefore there are two leaves, $p$ and $q$, such that $M(P) = M(Q)$ where $P$ is the $rp$-path and $Q$ is the $rq$-path. As illustrated in Figure \ref{fig:noDegreeEdge}.
	\begin{figure}[h]
		\begin{center}
			\begin{tikzpicture}
				[line width=1.4pt,vertex/.style={circle,inner sep=0pt,minimum size=0.15cm}, scale = 0.085]
				
				\node[draw = black, fill = black] (v) at     ($(0,0)$) [vertex] {};
				
				\node[draw = black, fill = black] (v1) at    ($(20,  -10)$) [vertex] {};
				\node[draw = black, fill = black] (v0) at    ($(-20, -10)$) [vertex] {};
				
				\node[draw = black, fill = black] (v11) at   ($(30,  -20)$) [vertex] {};
				\node[draw = black, fill = black] (v10) at   ($(10,  -20)$) [vertex] {};
				\node[draw = black, fill = black] (v01) at   ($(-10, -20)$) [vertex] {};
				\node[draw = black, fill = black] (v00) at   ($(-30, -20)$) [vertex] {};
				
				\node[draw = black, fill = black] (v111) at  ($(35,  -30)$) [vertex] {};
				\node[draw = black, fill = black] (v110) at  ($(25,  -30)$) [vertex] {};
				\node[draw = black, fill = black] (v101) at  ($(15,  -30)$) [vertex] {};
				\node[draw = black, fill = black] (v100) at  ($(5,   -30)$) [vertex] {};
				\node[draw = black, fill = black] (v011) at  ($(-5,  -30)$) [vertex] {};
				\node[draw = black, fill = black] (v010) at  ($(-15, -30)$) [vertex] {};
				\node[draw = black, fill = black] (v001) at  ($(-25, -30)$) [vertex] {};
				\node[draw = black, fill = black] (v000) at  ($(-35, -30)$) [vertex] {};
				
				\node[draw = black, fill = black] (v1111) at ($(37.5,  -40)$) [vertex] {};
				\node[draw = black, fill = black] (v1110) at ($(32.5,  -40)$) [vertex] {};
				\node[draw = black, fill = black] (v1101) at ($(27.5,  -40)$) [vertex] {};
				\node[draw = black, fill = black] (v1100) at ($(22.5,  -40)$) [vertex] {};
				\node[draw = black, fill = black] (v1011) at ($(17.5,  -40)$) [vertex] {};
				\node[draw = black, fill = black] (v1010) at ($(12.5,  -40)$) [vertex] {};
				\node[draw = black, fill = black] (v1001) at ($(7.5,   -40)$) [vertex] {};
				\node[draw = black, fill = black] (v1000) at ($(2.5,   -40)$) [vertex] {};
				\node[draw = black, fill = black] (v0111) at ($(-2.5,  -40)$) [vertex] {};
				\node[draw = black, fill = black] (v0110) at ($(-7.5,  -40)$) [vertex] {};
				\node[draw = black, fill = black] (v0101) at ($(-12.5, -40)$) [vertex] {};
				\node[draw = black, fill = black] (v0100) at ($(-17.5, -40)$) [vertex] {};
				\node[draw = black, fill = black] (v0011) at ($(-22.5, -40)$) [vertex] {};
				\node[draw = black, fill = black] (v0010) at ($(-27.5, -40)$) [vertex] {};
				\node[draw = black, fill = black] (v0001) at ($(-32.5, -40)$) [vertex] {};
				\node[draw = black, fill = black] (v0000) at ($(-37.5, -40)$) [vertex] {};

				\node[draw = white, fill = white  ] (p) at ($(-22.5, -44)$) [vertex] {$p$};
				\node[draw = white, fill = white  ] (q) at ($(-7.5,  -44)$) [vertex] {$q$};
				
				\draw[draw=yellow, line width = 5pt](v001)--(v0011);
				\draw[draw=yellow, line width = 5pt](v00)--(v001);
				\draw[draw=yellow, line width = 5pt](v0)--(v00);
				\draw[draw=yellow, line width = 5pt](v0)--(v01);
				\draw[draw=yellow, line width = 5pt](v01)--(v011);
				\draw[draw=yellow, line width = 5pt](v011)--(v0110);
				
				\draw[draw=red  ](v)--(v1);
				\draw[draw=blue ](v)--(v0);
				
				\draw[draw=green](v1)--(v11);
				\draw[draw=blue ](v1)--(v10);
				\draw[draw=red  ](v0)--(v01);
				\draw[draw=green](v0)--(v00);
				
				\draw[draw=blue ](v11)--(v111);
				\draw[draw=green](v11)--(v110);
				\draw[draw=green](v10)--(v101);
				\draw[draw=red  ](v10)--(v100);
				\draw[draw=blue ](v01)--(v011);
				\draw[draw=red  ](v01)--(v010);
				\draw[draw=blue ](v00)--(v001);
				\draw[draw=green](v00)--(v000);
				
				\draw[draw=green](v111)--(v1111);
				\draw[draw=blue ](v111)--(v1110);
				\draw[draw=red  ](v110)--(v1101);
				\draw[draw=blue ](v110)--(v1100);
				\draw[draw=green](v101)--(v1011);
				\draw[draw=blue ](v101)--(v1010);
				\draw[draw=red  ](v100)--(v1001);
				\draw[draw=red  ](v100)--(v1000);
				\draw[draw=blue ](v011)--(v0111);
				\draw[draw=green](v011)--(v0110);
				\draw[draw=green](v010)--(v0101);
				\draw[draw=green](v010)--(v0100);
				\draw[draw=red  ](v001)--(v0011);
				\draw[draw=blue ](v001)--(v0010);
				\draw[draw=green](v000)--(v0001);
				\draw[draw=red  ](v000)--(v0000);
				
			\end{tikzpicture}
		\end{center}
		\caption{The complete binary tree of height $4$ with an edge $3$-colouring. The leaves $p$ and $q$ have the same associated colour multiset so the $pq$-path is an anagram.
		}
		\label{fig:noDegreeEdge}
	\end{figure}
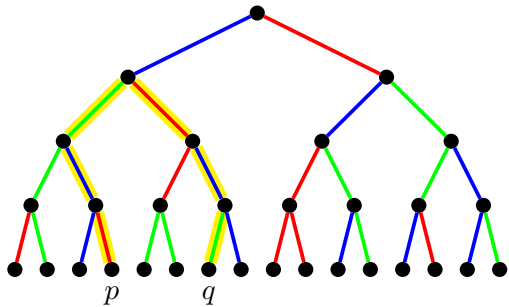
	
		Let $v$ be the least common ancestor of $p$ and $q$. Split these paths into three disjoint parts by defining $R$ as the $rv$-path, $P'$ as the $vp$-path and $Q'$ as the $vq$-path. Note that $E(R)$ is exactly the set of edges shared by $P$ and $Q$. Therefore
		\begin{align*}
			M(P) = M(P') \cup M(R) \text{ and }	M(Q) = M(Q') \cup M(R).
		\end{align*}
		Thus $M(P') = M(Q')$. Finally, note that $P' \cup Q'$ is a path and so $T$ contains an anagram.
	\end{proof}
	
	\subsection{Vertex colouring}
	
	Using the line graph construction, $\phi$ is unbounded on graphs of maximum degree $4$ as a corollary of Theorem \ref{thm:noDegreeBoundEdge}. The \textit{line graph} of a graph $G$, denoted $L(G)$, is the graph with $V(L(G))=E(G)$ and an edge between vertices of $L(G)$ which are incident to a common vertex in $G$. It is well known that for ordinary graph colouring $\chi'(G) = \chi(L(G))$. We have a similar relation for anagram-free colouring, however, equality does necessarily not hold. Line graphs satisfy $\phi'(G) \leq \phi(L(G))$ because every paths edge sequence $G$ correspond to a path $L(G)$. Theorem \ref{thm:noDegreeBoundEdge} says $\phi'$ is unbounded  for trees of maximum degree $3$ and, for these trees, $L(T)$ has maximum degree at most $4$. Also, as we have just shown, $\phi'(T) \leq \phi(L(T))$. Therefore $\phi$ is unbounded on graphs of maximum degree $4$. Ker{\"a}nen's result \cite{keranen1992abelian} implies $\phi(G) \leq 5$ for graphs of maximum degree $2$, thus there is a gap which motivates study of $\phi$ on graphs of maximum degree $3$.
	
	We prove Theorem \ref{thm:noDegreeBound}, that $\phi$ is unbounded on graphs of maximum degree $3$, with a method similar to the proof of Theorem \ref{thm:noDegreeBoundEdge}. We construct a graph with an exponential number of leaves, associate each leaf to a colour multiset and show that the graph contains an anagram if two leaves share a colour multiset.
	
	\begingroup
	\def\thetheorem{\ref{thm:noDegreeBound}}
	\begin{theorem}
		Graphs of maximum degree $3$ have unbounded anagram-free chromatic number.
	\end{theorem}
	\addtocounter{theorem}{-1}
	\endgroup
	
	\begin{proof}
	
		Let $c\geq 1$ and let $h \in \mathbb{Z}^+$ be odd such that $2^{(h + 1)/2} > (h + 2)^c$. Let $T$ be the rooted tree, with root $r$, such that:
		\begin{itemize}
			\item vertices of depth $h$ are leaves,
			\item vertices of even depth have two children,
			\item non-leaf vertices of odd depth have one child,
		\end{itemize}
		where the \textit{depth} of a vertex is its distance from the root.
		Let $G$ be the graph obtained from $T$ by adding an edge between every pair of vertices in $T$ that share a parent, as illustrated in Figure \ref{fig:noDegreeVertex}. 
	
		\begin{figure}[h]
			\begin{center}
				\begin{tikzpicture}
					[line width=1.4pt,vertex/.style={circle,inner sep=0pt,minimum size=0.3cm}, scale = 0.085]
					
					\node[draw = black, fill = green] (w) at     ($(0,0)$) [vertex] {};
					
					\node[draw = black, fill = red  ] (v1) at    ($(20,  -8)$) [vertex] {};
					\node[draw = black, fill = blue ] (v0) at    ($(-20, -8)$) [vertex] {};
					
					\node[draw = black, fill = red ] (w1) at    ($(20,  -16)$) [vertex] {};
					\node[draw = black, fill = green] (w0) at    ($(-20, -16)$) [vertex] {};
					
					\node[draw = black, fill = blue ] (v11) at   ($(30,  -24)$) [vertex] {};
					\node[draw = black, fill = green] (v10) at   ($(10,  -24)$) [vertex] {};
					\node[draw = black, fill = red  ] (v01) at   ($(-10, -24)$) [vertex] {};
					\node[draw = black, fill = blue ] (v00) at   ($(-30, -24)$) [vertex] {};
					
					\node[draw = black, fill = green] (w11) at   ($(30,  -32)$) [vertex] {};
					\node[draw = black, fill = blue ] (w10) at   ($(10,  -32)$) [vertex] {};
					\node[draw = black, fill = blue ] (w01) at   ($(-10, -32)$) [vertex] {};
					\node[draw = black, fill = red  ] (w00) at   ($(-30, -32)$) [vertex] {};
					
					\node[draw = black, fill = blue ] (v111) at  ($(35,  -40)$) [vertex] {};
					\node[draw = black, fill = red  ] (v110) at  ($(25,  -40)$) [vertex] {};
					\node[draw = black, fill = red  ] (v101) at  ($(15,  -40)$) [vertex] {};
					\node[draw = black, fill = green] (v100) at  ($(5,   -40)$) [vertex] {};
					\node[draw = black, fill = blue ] (v011) at  ($(-5,  -40)$) [vertex] {};
					\node[draw = black, fill = red  ] (v010) at  ($(-15, -40)$) [vertex] {};
					\node[draw = black, fill = blue ] (v001) at  ($(-25, -40)$) [vertex] {};
					\node[draw = black, fill = green] (v000) at  ($(-35, -40)$) [vertex] {};
					
					\node[draw = white, fill = white  ] (p) at ($(-15, -44)$) [vertex] {$p$};
					\node[draw = white, fill = white  ] (q) at ($(35,  -44)$) [vertex] {$q$};
					
					\draw[draw=yellow, line width = 5pt](w01)--(v010);
					\draw[draw=yellow, line width = 5pt](v01)--(w01);
					\draw[draw=yellow, line width = 5pt](w0)--(v01);
					\draw[draw=yellow, line width = 5pt](v0)--(w0);
					\draw[draw=yellow, line width = 5pt](v1)--(v0);
					\draw[draw=yellow, line width = 5pt](v1)--(w1);
					\draw[draw=yellow, line width = 5pt](w1)--(v11);
					\draw[draw=yellow, line width = 5pt](v11)--(w11);
					\draw[draw=yellow, line width = 5pt](w11)--(v111);
					
					\draw[draw=black](w)--(v1);
					\draw[draw=black](w)--(v0);
					
					\draw[draw=black](v1)--(v0);
					
					\draw[draw=black](v1)--(w1);
					\draw[draw=black](v0)--(w0);
					
					\draw[draw=black](w1)--(v11);
					\draw[draw=black](w1)--(v10);
					\draw[draw=black](w0)--(v01);
					\draw[draw=black](w0)--(v00);
					
					\draw[draw=black](v01)--(v00);
					\draw[draw=black](v11)--(v10);
					
					\draw[draw=black](v11)--(w11);
					\draw[draw=black](v11)--(w11);
					\draw[draw=black](v10)--(w10);
					\draw[draw=black](v10)--(w10);
					\draw[draw=black](v01)--(w01);
					\draw[draw=black](v01)--(w01);
					\draw[draw=black](v00)--(w00);
					\draw[draw=black](v00)--(w00);
					
					\draw[draw=black](v101)--(v100);
					\draw[draw=black](v111)--(v110);
					\draw[draw=black](v001)--(v000);
					\draw[draw=black](v011)--(v010);
					
					\draw[draw=black](w11)--(v111);
					\draw[draw=black](w11)--(v110);
					\draw[draw=black](w10)--(v101);
					\draw[draw=black](w10)--(v100);
					\draw[draw=black](w01)--(v011);
					\draw[draw=black](w01)--(v010);
					\draw[draw=black](w00)--(v001);
					\draw[draw=black](w00)--(v000);
				\end{tikzpicture}
			\end{center}
			\caption{The graph $G$ with $h = 5$, a vertex $3$-colouring. Vertces $p$ and $q$ have the same colour multiset so the $pq$-path is an anagram.
			}
			\label{fig:noDegreeVertex}
		\end{figure}
		Fix an arbitrary $c$-colouring of $G$. We now show that $G$ contains an anagram. By Equation (\ref{eqn:colourMult}) and our choice of $h$,
		\begin{align*}
			\#\text{leaves of $T$} = 2^{(h+1)/2} > (h+2)^c \geq |\mathcal{M}_{h+1, c}|.
		\end{align*}
		So, because each root-to-leaf path in $T$ has $h+1$ vertices, the number of leaves in $T$ is greater than the number of distinct colour multisets on root-to-leaf paths in $T$. Therefore there are two leaves, $p$ and $q$ in $T$, such that $M(P)=M(Q)$ where $P$ and $Q$ are the $rp$-path in $T$ and $rq$-path in $T$ respectively. Split the two paths into three vertex-disjoint paths $R = P \cap Q$, $P' = P - V(R)$ and $Q' = Q - V(R)$. Note that
		\begin{align*}
			M(P) = M(P') \cup M(R) \text{ and } M(Q) = M(Q') \cup M(R).
		\end{align*}
		Thus $M(P') = M(Q')$. By construction, $G\left[V(P') \cup V(Q')\right]$ is a path. Therefore $G$ contains an anagram.
	\end{proof}
	Theorem \ref{thm:noDegreeBound} only proves that $\phi$ is unbounded on graphs of maximum degree $3$ without further reference to their structure. We prove below that $\phi$ is unbounded on trees. A particularly interesting question is whether there is a result analogous to Theorem \ref{thm:noDegreeBoundEdge}: is $\phi$ bounded on trees of maximum degree $3$? This motivates further investigation of $\phi$ on trees. 
	
	\subsection{Vertex colouring trees}\label{sec:lowerTree}
	We now prove lower bounds for anagram-free vertex colourings of trees. These bounds follow from the following theorem. The \textit{complete $d$-ary tree of height $h$} is the rooted tree such that every internal vertex has $d$ children and all leaves are of distance $h$ from the root.
			
	\begin{theorem}\label{thm:treeCompleteDary}
		The complete $d$-ary tree of height $h$ does not have an anagram-free $c$-colouring when $d^c \leq (d/c)^h$.
	\end{theorem}
	\begin{proof}
		Let $T$ be the complete $d$-ary tree of height $h$ with root $r$. Let $L$ be the leaves of $T$ and fix an arbitrary $c$-colouring of $T$.
		
		For each $v \in L$ let $S_v$ be the sequence of colours on the $rv$-path. There are at most $c^h$ such colour sequences since each has length $h+1$ and they share an initial colour (the colour of the root). Since $|L| = d^h$ there is a set $C \subseteq L$ of size at least $d^h / c^h$ such that $S_v = S_w$ for all $v,w \in C$. Thus $C$ is a large set of leaves that have equal colour sequence on each root-to-leaf path. As illistrated in Figure \ref{fig:completeDaryTree}.
		
		Let $R$ be the subtree of $T$ induced by the set of all ancestors of leaves in $C$. The remainder of the proof is concerned with finding an anagram in $R$. Observe that if two vertices of $R$ have the same depth, then they also have the same colour. Define a \textit{level} as a maximal set of vertices in $R$ that all have equal depth. $R$ is coloured by level. Denote the levels $l_0, l_1, \ldots, l_{h}$ where $l_0 = \{r\}$.
			
		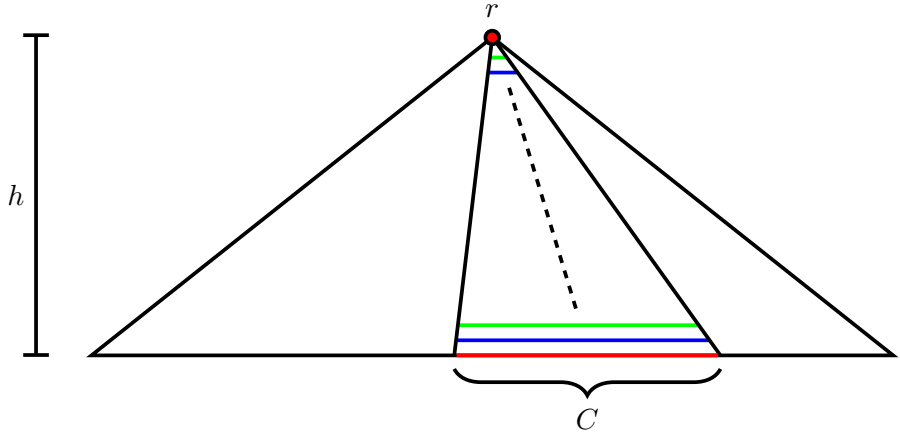
\begin{figure}[h]
			\begin{center}
				\begin{tikzpicture}[
						line width = 1.4pt, 
						vertex/.style={circle,inner sep=0pt,minimum size=0.5cm}, 
						sVertex/.style={circle,inner sep=0pt,minimum size=0.18cm}, 
						point/.style={circle,inner sep=0pt,minimum size=0cm}, 
						scale = 0.1
					]

					\draw [|-|, left] ($(-60, -18.4)$) -- ($(-60, 24.5)$) node[midway, color=black] {$h$};
					
					\node (tri) [isosceles triangle, draw, isosceles triangle stretches, minimum width = 300, minimum height = 120, shape border rotate=90] at ($(0, 0)$) {};
					
					\draw [color=green] ($(2, 21.3)$) -- ($(-0.2, 21.3)$) {};
					\draw [color=blue ] ($(3.3, 19.3)$) -- ($(-0.4, 19.3)$) {};
					
					\draw [color=blue ] ($(28.6, -16.2)$) -- ($(-4.8, -16.2)$) {};
					\draw [color=green] ($(27, -14.2)$) -- ($(-4.6, -14.2)$) {};
					
					\draw [] ($(-5, -18.2)$) -- ($(0, 24)$) node[midway, color=black] {};
					\draw [] ($(30, -18.2)$) -- ($(0, 24)$) node[midway, color=black] {};
					
					\draw [dashed] ($(11, -12)$) -- ($(2, 18)$) node[midway, color=black] {};
					
					\node[draw = black, fill = red, label = {[label distance=0pt]90:$r$}] (r) at ($(0, 23.95)$) [sVertex] {};
				
					\draw [color=red  ] ($(29.7, -18.2)$) -- ($(-4.7, -18.2)$) {};
					
					\draw [decorate,decoration={brace,amplitude=10pt},xshift=0pt,yshift=0pt] ($(30, -20)$) -- ($(-5, -20)$) node[midway, color=black, yshift= -19pt] {$C$};
					
				\end{tikzpicture}
			\end{center}
			\caption{The complete $d$-ary tree of height $h$ with a large set $C$ of leaves which have the same root-to-leaf path colour sequence.}
			\label{fig:completeDaryTree}
		\end{figure}
		
		A level of $R$ is \textit{bad} if every vertex in the level has exactly one child in $R$. A level is \textit{good} if it is not bad. Note that only level $l_h$ contains vertices with no children. Let $g$ be the number of good levels and $b$ be the number of bad levels. Then $h+1 = g+b$. We now prove that there are at least $c+1$ good levels and so at least two good levels share a colour.
		
		We bound the number of bad levels by considering the number of good levels required to obtain at least $(d/c)^h$ leaves. If $l_i$ is bad then $|l_i| = |l_{i+1}|$ and if $l_i$ is good then $|l_i| < |l_{i+1}| \leq d|l_i|$. 
		\begin{align*}
		|l_i| \leq d^\text{\#preceding good levels}
		\end{align*}
		since $l_h$ is the final good level, it is preceded by $g-1$ good levels. Thus
		\begin{align*}
		(d/c)^h \leq|l_h| \leq d^{g-1}
		\end{align*}
		By assumption, $d^{c+1} \leq d^{h+1}c^{-h}$ which is at most $d^g$. Thus $c+1 \leq g$, there are two good levels with the same colour.
		
		Let $l_a$ and $l_b$ be two good levels that have the same colour and, without loss of generality, let $a < b$. Let $v \in l_a$ be a vertex with at least two children. All vertices in the levels between $l_a$ and $l_b$ have at least one child so there are two vertices $u, w \in B$ such that $v$ is their least common ancestor. 
		
		Let $p_0, p_1, \ldots, p_n$ be the $uv$-path and $q_0, q_1, \ldots, q_n$ be the $wv$-path. Since  $R$ is coloured by level,
		\begin{align*}
		M(p_0, p_1, \ldots, p_n)=M(q_0, q_1, \ldots, q_n).
		\end{align*}
		Vertices $p_0$ and $q_n$ are in levels $l_a$ and $l_b$ respectively so they have the same colour. Thus
		\begin{align*}
		M(p_1, p_2, \ldots, p_n)=M(q_0, q_1, \ldots, q_{n-1}).
		\end{align*}
		Therefore $p_0, p_1, \ldots, p_{n-1}, p_n, q_{n-1}, \ldots, q_1$ is an anagram.
	\end{proof}
	Theorem \ref{thm:treeCompleteDary} implies Theorem \ref{thm:treeDegreeBound}, that $\phi$ is unbounded on trees, but it has uses beyond this result because it gives us control over the radius and maximum degree of the resulting trees. In particular, it implies the following two theorems each of which correspond to an extreme of radius or maximum degree.
	
	\begin{theorem}\label{thm:treeLowerHeight}
		For every integer $h \geq 2$ the complete $(h - 1)^h$-ary tree of height $h$ has $\phi(T) \geq h$.
	\end{theorem}
	\begin{proof}
		Let $h := c + 1$ and $d := c^{c + 1}$. The conditions of Theorem \ref{thm:treeCompleteDary} are satisfied because
		\begin{align*}
		d^c = c^{c(c+1)} = \lr{c^{c(c+1)}c^{c+1}}c^{-(c+1)} = d^{c+1}c^{-(c+1)} = (d/c)^h.
		\end{align*}
		Therefore the complete $c^{c + 1}$-ary tree of height $c+1$ does not have an anagram-free $c$-colouring.
	\end{proof}
	\begin{theorem}\label{thm:treeLowerDegree}
		For every integer $d' \geq 1$ there exists a tree $T$ with maximum degree $d'$ such that $\phi(T) \geq d' - 1$.
	\end{theorem}
	\begin{proof}
		Fix $c$ and choose $h$ so that
		\begin{align*}
		(c+1)^c \leq (1+ \frac{1}{c})^h = (\frac{c+1}{c})^h.
		\end{align*}
		Thus Theorem \ref{thm:treeCompleteDary} is satisfied with $d = c+1$. Therefore the complete $(c+1)$-ary tree of height $h$ does not have an anagram-free $c$-colouring. This tree has maximum degree $d' = c+2$ and so $\phi(T) \geq c + 1 = d' - 1$.
	\end{proof}
	
	\section{Upper bounds for $\phi$ on trees}\label{sec:upperTree}
	In this section we complement the results of the previous section with some upper bounds for $\phi$ on trees. Our first bound comes from centred colouring. A vertex colouring of a graph $G$ is \textit{centred} if every subtree $T$ of $G$ contains a vertex whose colour appears exactly once in $T$. Centred colourings are anagram-free since every anagram contains an even number of occurrences of each colour and in a centred colouring every path contains a colour that occurs exactly once. Therefore, for every graph $G$, the centred chromatic number of $G$ is an upper bound on $\phi(G)$. It is easily seen that every tree $T$ of radius $h$ has centred chromatic number at most $h+1$, see \cite[Section 6.5]{book:sparsity}, thus $\phi(T) \leq h + 1$. This bound completes the first half of Theorem \ref{thm:treeHeightBound} and is attained by colouring each vertex by its distance from a centre of $T$. The second half of Theorem \ref{thm:treeHeightBound} follows from Theorem \ref{thm:treeLowerHeight} and an inspection of the trees of radius $0$ or $1$.
	
	We now work towards proving Theorem \ref{thm:treePathwidthBound}, which bounds $\phi$ by a function of pathwidth. While pathwidth is defined in terms of path decompositions, the only property we require is the following lemma.
	\begin{lemma}[\cite{suderman2004pathwidth} Lemma 5]\label{lem:mainPath}
		Every tree $T$ with at least one edge contains a path $P$ such $\pw(T - V(P)) \leq \pw(T)-1$.
	\end{lemma}
	We also require two trivial properties of pathwidth. The first is that edgeless graphs have pathwidth $0$ and the second is that the pathwidth of a disconnected graph is the maximum pathwidth over each of its components.
	\begin{theorem}\label{thm:colourPathwidth}
		Every tree of pathwidth $m$ has an anagram-free vertex $(4m+1)$-colouring.
	\end{theorem}
	\begin{proof}
		The proof is an induction on $m$. The base case holds because every tree $T$ of pathwidth $0$ is edgeless and thus anagram-free $1$-colourable. Now assume that every tree with pathwidth $p \leq m$ is anagram-free vertex $(4p+1)$-colourable.
		
		Let $T$ be a tree of pathwidth $m + 1$. By Lemma \ref{lem:mainPath}, there exists a path $P \subseteq T$ such that $\pw(T - V(P)) \leq m$. Each component of $T - V(P)$ has pathwidth at most $m$. By induction we may anagram-free colour each component of $T - V(P)$ with a common set of $4m+1$ colours. Then use four additional colours to anagram-free colour $P$, by \cite{keranen2009abelian}.
		
		We now show that this colouring is anagram-free. Let $Q$ be a path in $T$. If $Q$ is entirely contained within a component of $T - V(P)$, then by induction, $Q$ is not an anagram. Otherwise, $Q$ intersects $P$. The intersection of $Q$ and $P$ is an anagram-free subpath of $P$ and the colours in $Q \cap P$ occur nowhere else in $Q$. Therefore $Q$ is not an anagram.
	\end{proof}
	We now use Theorems \ref{thm:colourPathwidth} and \ref{thm:treeLowerHeight} to prove Theorem \ref{thm:treePathwidthBound}.
	\begingroup
	\def\thetheorem{\ref{thm:treePathwidthBound}}
	\begin{theorem}
		For every tree $T$, $\phi(T) \leq 4\pw(T) + 1$. Moreover, for every $p \geq 0$ there is a tree $T$ such that $\phi(T) \geq p \geq \pw(T)$.
	\end{theorem}
	\addtocounter{theorem}{-1}
	\endgroup
	\begin{proof}
		The first part follows directly from Theorem \ref{thm:colourPathwidth}. For the second part it is well known, and easily proved, that the pathwidth of a tree is at most its radius. Therefore, by Theorem \ref{thm:treeLowerHeight}, there exists a tree $T$ with $\phi(T) \geq p \geq \pw(T)$ for all $p \geq 2$. For the remaining cases use the path of order $2$ for $p=1$ and the empty graph for $p = 0$.
	\end{proof}
	
	The main open problem that arises from the above results is whether $\phi$ is bounded for trees of maximum degree 3. The complete binary tree of height $h$ is the key example. Centred colourings provides a trivial upper bound of $h+1$ (Theorem \ref{thm:treeHeightBound}). We now give a non-trivial colouring of the complete binary tree to demonstrate that the bound from centred colouring is not best possible.
	
	\begin{theorem}\label{thm:treeBinaryHeightBound}
		If $T$ is the complete binary tree of height $h$, then 
		\begin{align}\label{eqn:phiTsingle}
			\phi(T) \leq \frac{h}{2} + \frac{1}{2}\log_2(h+1) + 1.
		\end{align}
	\end{theorem}
	\begin{proof}
		We proceed by induction on $h$. The base case is satisfied as follows:
		\begin{align*}
			\phi(\text{singleton vertex}) = 1 \leq \frac{0}{2} + \frac{1}{2} \log_2(0 + 1) + 1 = 1.
		\end{align*}
		
		Now assume the result holds up to $h-1$. Let $T$ be the complete binary tree of height $h$ with root $r$. Let $t:= 2\floor*{\frac{h + 2}{4}}$ and $T_t \subseteq T$ be the complete binary tree of height $t$ with root $r$. $T_t$ is the top half of $T$, which we colour directly. Let $b := h - t - 1$, this is the height of each subtree which will be coloured by induction. Colour $T_t$ as follows:
		\begin{itemize}
			\item All vertices with even depth receive the same colour. Call this colour $c$.
			\item Each odd level is allocated distinct set of two colours. Vertices of odd depth are coloured with one of the two colours allocated to their level, so that each vertex receives a colour different from their sibling.
		\end{itemize}
		Note that the leaves of $T_t$ have even depth so have colour $c$. This colouring is shown for $h = 8$ in Figure \ref{fig:binaryTreeVertex}.
		
		\begin{figure}[h]
			\begin{center}
				\begin{tikzpicture}[
						line width = 1.4pt, 
						vertex/.style={circle,inner sep=0pt,minimum size=0.5cm}, 
						sVertex/.style={circle,inner sep=0pt,minimum size=0.15cm}, 
						point/.style={circle,inner sep=0pt,minimum size=0cm}, 
						scale = 0.15
					]
					
					\node[draw = black, fill = white] (v) at     ($(0,0)$) [vertex] {$1$};
					
					\node[draw = black, fill = white] (v1) at    ($(20,  -5)$) [vertex] {$3$};
					\node[draw = black, fill = white] (v0) at    ($(-20, -5)$) [vertex] {$2$};
					
					\node[draw = black, fill = white] (v11) at   ($(30,  -10)$) [vertex] {$1$};
					\node[draw = black, fill = white] (v10) at   ($(10,  -10)$) [vertex] {$1$};
					\node[draw = black, fill = white] (v01) at   ($(-10, -10)$) [vertex] {$1$};
					\node[draw = black, fill = white] (v00) at   ($(-30, -10)$) [vertex] {$1$};
					
					\node[draw = black, fill = white] (v111) at  ($(35,  -15)$) [vertex] {$5$};
					\node[draw = black, fill = white] (v110) at  ($(25,  -15)$) [vertex] {$4$};
					\node[draw = black, fill = white] (v101) at  ($(15,  -15)$) [vertex] {$5$};
					\node[draw = black, fill = white] (v100) at  ($(5,   -15)$) [vertex] {$4$};
					\node[draw = black, fill = white] (v011) at  ($(-5,  -15)$) [vertex] {$5$};
					\node[draw = black, fill = white] (v010) at  ($(-15, -15)$) [vertex] {$4$};
					\node[draw = black, fill = white] (v001) at  ($(-25, -15)$) [vertex] {$5$};
					\node[draw = black, fill = white] (v000) at  ($(-35, -15)$) [vertex] {$4$};
					
					\node[draw = black, fill = white] (v1111) at ($(37.5,  -20)$) [vertex] {$1$};
					\node[draw = black, fill = white] (v1110) at ($(32.5,  -20)$) [vertex] {$1$};
					\node[draw = black, fill = white] (v1101) at ($(27.5,  -20)$) [vertex] {$1$};
					\node[draw = black, fill = white] (v1100) at ($(22.5,  -20)$) [vertex] {$1$};
					\node[draw = black, fill = white] (v1011) at ($(17.5,  -20)$) [vertex] {$1$};
					\node[draw = black, fill = white] (v1010) at ($(12.5,  -20)$) [vertex] {$1$};
					\node[draw = black, fill = white] (v1001) at ($(7.5,   -20)$) [vertex] {$1$};
					\node[draw = black, fill = white] (v1000) at ($(2.5,   -20)$) [vertex] {$1$};
					\node[draw = black, fill = white] (v0111) at ($(-2.5,  -20)$) [vertex] {$1$};
					\node[draw = black, fill = white] (v0110) at ($(-7.5,  -20)$) [vertex] {$1$};
					\node[draw = black, fill = white] (v0101) at ($(-12.5, -20)$) [vertex] {$1$};
					\node[draw = black, fill = white] (v0100) at ($(-17.5, -20)$) [vertex] {$1$};
					\node[draw = black, fill = white] (v0011) at ($(-22.5, -20)$) [vertex] {$1$};
					\node[draw = black, fill = white] (v0010) at ($(-27.5, -20)$) [vertex] {$1$};
					\node[draw = black, fill = white] (v0001) at ($(-32.5, -20)$) [vertex] {$1$};
					\node[draw = black, fill = white] (v0000) at ($(-37.5, -20)$) [vertex] {$1$};
						
					\node[draw = black, fill = black] (v11111) at ($(37.5  + 3,  -25)$)   [sVertex] {};
					
					\node[draw = black, fill = black] (v11110) at ($(32.5  + 1, -24)$)   [point] {};
					\node[draw = black, fill = black] (v11101) at ($(27.5  + 1, -24)$)   [point] {};
					\node[draw = black, fill = black] (v11100) at ($(22.5  + 1, -24)$)   [point] {};
					\node[draw = black, fill = black] (v11011) at ($(17.5  + 1, -24)$)   [point] {};
					\node[draw = black, fill = black] (v11010) at ($(12.5  + 1, -24)$)   [point] {};
					\node[draw = black, fill = black] (v11001) at ($(7.5   + 1, -24)$)   [point] {};
					\node[draw = black, fill = black] (v11000) at ($(2.5   + 1, -24)$)   [point] {};
					\node[draw = black, fill = black] (v10111) at ($(-2.5  + 1, -24)$)   [point] {};
					\node[draw = black, fill = black] (v10110) at ($(-7.5  + 1, -24)$)   [point] {};
					\node[draw = black, fill = black] (v10101) at ($(-12.5 + 1, -24)$)   [point] {};
					\node[draw = black, fill = black] (v10100) at ($(-17.5 + 1, -24)$)   [point] {};
					\node[draw = black, fill = black] (v10011) at ($(-22.5 + 1, -24)$)   [point] {};
					\node[draw = black, fill = black] (v10010) at ($(-27.5 + 1, -24)$)   [point] {};
					\node[draw = black, fill = black] (v10001) at ($(-32.5 + 1, -24)$)   [point] {};
					
					\node[draw = black, fill = black] (v10000) at ($(-37.5 + 3, -25)$) [sVertex] {};
					
					\node[draw = black, fill = black] (v01111) at ($(37.5  - 3, -25)$) [sVertex] {};
					
					\node[draw = black, fill = black] (v01110) at ($(32.5  - 1, -24)$)   [point] {};
					\node[draw = black, fill = black] (v01101) at ($(27.5  - 1, -24)$)   [point] {};
					\node[draw = black, fill = black] (v01100) at ($(22.5  - 1, -24)$)   [point] {};
					\node[draw = black, fill = black] (v01011) at ($(17.5  - 1, -24)$)   [point] {};
					\node[draw = black, fill = black] (v01010) at ($(12.5  - 1, -24)$)   [point] {};
					\node[draw = black, fill = black] (v01001) at ($(7.5   - 1, -24)$)   [point] {};
					\node[draw = black, fill = black] (v01000) at ($(2.5   - 1, -24)$)   [point] {};
					\node[draw = black, fill = black] (v00111) at ($(-2.5  - 1, -24)$)   [point] {};
					\node[draw = black, fill = black] (v00110) at ($(-7.5  - 1, -24)$)   [point] {};
					\node[draw = black, fill = black] (v00101) at ($(-12.5 - 1, -24)$)   [point] {};
					\node[draw = black, fill = black] (v00100) at ($(-17.5 - 1, -24)$)   [point] {};
					\node[draw = black, fill = black] (v00011) at ($(-22.5 - 1, -24)$)   [point] {};
					\node[draw = black, fill = black] (v00010) at ($(-27.5 - 1, -24)$)   [point] {};
					\node[draw = black, fill = black] (v00001) at ($(-32.5 - 1, -24)$)   [point] {};
					
					\node[draw = black, fill = black] (v00000) at ($(-37.5 - 3, -25)$)   [sVertex] {};
					
					\draw [|-|, left] ($(-45, -20)$) -- ($(-45, 0)$) node[midway, color=black] {$t = 4$};
					\draw [|-|, left] ($(-45, -31.2)$)  -- ($(-45, -25)$) node[midway, color=black] {$b = 3$};
					
					\node (tri) [isosceles triangle, draw, minimum width = 7, minimum height = 26, shape border rotate=90] at ($(-37.5 - 3, -29.5)$) {};
					\node (tri) [isosceles triangle, draw, minimum width = 7, minimum height = 26, shape border rotate=90] at ($(-37.5 + 3, -29.5)$) {};
					\node (tri) [isosceles triangle, draw, minimum width = 7, minimum height = 26, shape border rotate=90] at ($( 37.5 - 3, -29.5)$) {};
					\node (tri) [isosceles triangle, draw, minimum width = 7, minimum height = 26, shape border rotate=90] at ($( 37.5 + 3, -29.5)$) {};
					
					\node[draw = black, fill = black] (v111110) at ($(37.5  + 3 - 2.5, -32)$)  [point] {};
					\node[draw = black, fill = black] (v111111) at ($(37.5  + 3 + 2.5, -32)$)  [point] {};

					\node[draw = black, fill = black] (v100000) at ($(-37.5 + 3 - 2.5, -32)$)  [point] {};
					\node[draw = black, fill = black] (v100001) at ($(-37.5 + 3 + 2.5, -32)$)  [point] {};

					\node[draw = black, fill = black] (v011110) at ($(37.5  - 3 - 2.5, -32)$)  [point] {};
					\node[draw = black, fill = black] (v011111) at ($(37.5  - 3 + 2.5, -32)$)  [point] {};
					
					\node[draw = black, fill = black] (v000000) at ($(-37.5 - 3 - 2.5, -32)$)  [point] {};
					\node[draw = black, fill = black] (v000001) at ($(-37.5 - 3 + 2.5, -32)$)  [point] {};
					
					\node[opacity = 1, minimum size=0cm] (p) at ($(-37.5, -33)$) [vertex] {$\{3,5,6\}$};
					\node[opacity = 1, minimum size=0cm] (q) at ($( 37.5, -33)$) [vertex] {$\{2,4,6\}$};
					
					\draw[draw=black](v)--(v1);
					\draw[draw=black](v)--(v0);
					
					\draw[draw=black](v1)--(v11);
					\draw[draw=black](v1)--(v10);
					\draw[draw=black](v0)--(v01);
					\draw[draw=black](v0)--(v00);
					
					\draw[draw=black](v11)--(v111);
					\draw[draw=black](v11)--(v110);
					\draw[draw=black](v10)--(v101);
					\draw[draw=black](v10)--(v100);
					\draw[draw=black](v01)--(v011);
					\draw[draw=black](v01)--(v010);
					\draw[draw=black](v00)--(v001);
					\draw[draw=black](v00)--(v000);
					
					\draw[draw=black](v111)--(v1111);
					\draw[draw=black](v111)--(v1110);
					\draw[draw=black](v110)--(v1101);
					\draw[draw=black](v110)--(v1100);
					\draw[draw=black](v101)--(v1011);
					\draw[draw=black](v101)--(v1010);
					\draw[draw=black](v100)--(v1001);
					\draw[draw=black](v100)--(v1000);
					\draw[draw=black](v011)--(v0111);
					\draw[draw=black](v011)--(v0110);
					\draw[draw=black](v010)--(v0101);
					\draw[draw=black](v010)--(v0100);
					\draw[draw=black](v001)--(v0011);
					\draw[draw=black](v001)--(v0010);
					\draw[draw=black](v000)--(v0001);
					\draw[draw=black](v000)--(v0000);
					
					\draw[draw=black](v1111)--(v11111);
					
					\draw[draw=black, dashed](v1110)--(v11110);
					\draw[draw=black, dashed](v1101)--(v11101);
					\draw[draw=black, dashed](v1100)--(v11100);
					\draw[draw=black, dashed](v1011)--(v11011);
					\draw[draw=black, dashed](v1010)--(v11010);
					\draw[draw=black, dashed](v1001)--(v11001);
					\draw[draw=black, dashed](v1000)--(v11000);
					\draw[draw=black, dashed](v0111)--(v10111);
					\draw[draw=black, dashed](v0110)--(v10110);
					\draw[draw=black, dashed](v0101)--(v10101);
					\draw[draw=black, dashed](v0100)--(v10100);
					\draw[draw=black, dashed](v0011)--(v10011);
					\draw[draw=black, dashed](v0010)--(v10010);
					\draw[draw=black, dashed](v0001)--(v10001);
					
					\draw[draw=black](v0000)--(v10000);
					
					\draw[draw=black](v1111)--(v01111);
					
					\draw[draw=black, dashed](v1110)--(v01110);
					\draw[draw=black, dashed](v1101)--(v01101);
					\draw[draw=black, dashed](v1100)--(v01100);
					\draw[draw=black, dashed](v1011)--(v01011);
					\draw[draw=black, dashed](v1010)--(v01010);
					\draw[draw=black, dashed](v1001)--(v01001);
					\draw[draw=black, dashed](v1000)--(v01000);
					\draw[draw=black, dashed](v0111)--(v00111);
					\draw[draw=black, dashed](v0110)--(v00110);
					\draw[draw=black, dashed](v0101)--(v00101);
					\draw[draw=black, dashed](v0100)--(v00100);
					\draw[draw=black, dashed](v0011)--(v00011);
					\draw[draw=black, dashed](v0010)--(v00010);
					\draw[draw=black, dashed](v0001)--(v00001);
					
					\draw[draw=black](v0000)--(v00000);
				\end{tikzpicture}
			\end{center}
			\caption{Schematic of colouring the complete binary tree with $h = 8$. The set of colours on shown subtrees of height $b$ are shown below the trees.}
			\label{fig:binaryTreeVertex}
		\end{figure}
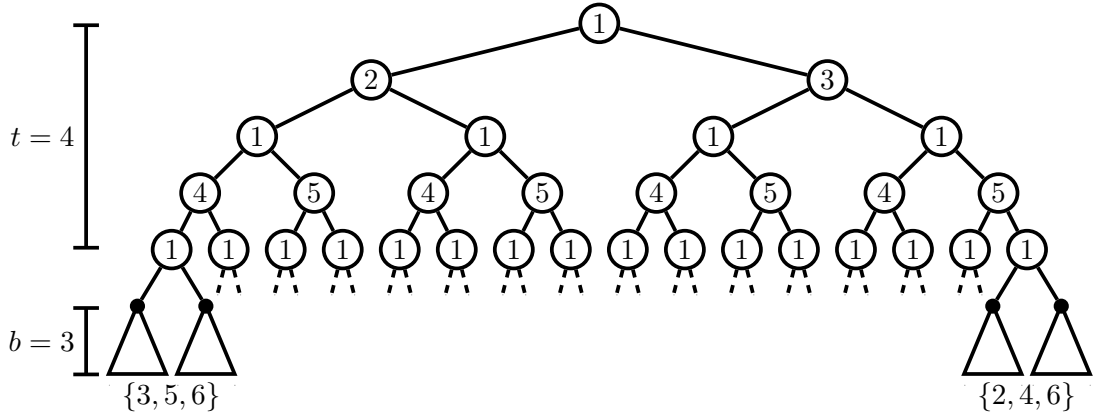
		
		Colour each remaining subtree of $T$ by induction, avoiding colours that occur on their ancestors in $T_t$.
		
		\begin{claim*}
			This colouring of $T$ is anagram-free.
		\end{claim*}
		\begin{proof}
			Let $P$ be a path in $T$ with even order at least $2$. Let $u$ be the shallowest vertex in $P$. If $u \notin V(T_t)$ then, by induction, $P$ is not an anagram. Now consider the case where $u \in V(T_t)$. If $u$ has odd depth then its colour is unique in $P$. Indeed, in $T_t$ the colour of $u$ only occurs in the level of $u$ and in $T - V(T_t)$ the ancestors of $u$ avoid its colour. Similarly, if $u$ has even depth and $u$ is an endpoint of $P$ then the child of $u$ in $P$ is uniquely coloured in $P$.
			
			The remainder of the proof is concerned with the case where $u \in V(T_t)$, $u$ has even depth and neither endpoint of $P$ is $u$. Let $x_1,x_2 \in V(P)$ be the endpoints of $P$ and let $v_1,v_2 \in V(P)$ be the children of $u$ such that $v_i$ is an ancestor of $x_i$ for $i \in [2]$. If both $x_1, x_2 \notin V(T_t)$ then $v_1$ (and indeed $v_2$) is uniquely coloured in $P$. In the remaining cases at least one of $x_1$ and $x_2$ are in $V(T_t)$.
			
			Without loss of generality let $x_1 \in V(T_t)$. If $x_2 \notin V(T_t)$ then $x_1$ is uniquely coloured in $P$. If $x_2 \in V(T_t)$ then the colour $c$ occurs an odd number of times in $P$. Indeed, $u$ has colour $c$ and $x_1,x_2 \in V(T_t)$ implies that both the $v_1x_1$-path and the $v_2x_2$-path contain the same number of vertices with colour $c$.
		\end{proof}
		
		To complete the proof we show that this colouring satisfies Equation (\ref{eqn:phiTsingle}). Our colouring of $T_t$ uses $t + 1$ colours as the even levels share a colour and there are $\frac{t}{2}$ odd levels which each use $2$ colours.
				
		Let $v$ be a child of a leaf of $T_t$ and $T_b$ be the subtree of $T$ rooted at $v$. Recall that colours on the $vr$-path do not occur in $T_b$. The number of distinct colours on the $vr$-path is $\frac{t}{2} + 1$ because the path contains $t+1$ vertices and $\frac{t}{2} + 1$ of them have even depth so share colour $c$. Therefore the colouring of $T_b$ can reuse $\frac{t}{2}$ of the colours used to colour $T_t$. So our colouring of $T$ requires $\phi(T_b) - \frac{t}{2}$ colours in addition to those used to colour $T_t$. So
		\begin{align*}
			\phi(T) \leq t + 1 + (\phi(T_b) - \frac{t}{2}) = \frac{t}{2} + 1 + \phi\lr{T_{h - t - 1}}.
		\end{align*}
		By induction, since $T_b$ is a complete binary tree of height $h-t-1$,
		\begin{align*}
				\phi(T) &\leq \frac{t}{2} + 1 + \lr{\frac{h - t - 1}{2} + \frac{1}{2} \log_2\lr{h - t} + 1} \\
				&= \frac{h}{2} + \frac{1}{2} + \frac{1}{2} \log_2\lr{h - t} + 1\\
				&= \frac{h}{2} + \frac{1}{2} + \frac{1}{2} \log_2\lr{h - 2\floor*{\frac{h + 2}{4}}} + 1\\
				&\leq \frac{h}{2} + \frac{1}{2} + \frac{1}{2} \log_2\lr{h - \frac{h - 1}{2}} + 1 \\
				&= \frac{h}{2} + \frac{1}{2} + \frac{1}{2} \log_2\lr{\frac{h + 1}{2}} + 1 \\
				&= \frac{h}{2} + \frac{1}{2} \log_2\lr{h + 1} + 1. \qedhere
		\end{align*}
	\end{proof}
	
	\section{$k$-anagram-free colourings}\label{sec:kPower}
	
	Recall that a $k$-anagram consists of $k$ independently permuted copies of a word. In terms of colour multisets, a word $W_1W_2\ldots W_k$ is a $k$-anagram if for $i,j \in [k]$,
	\begin{align*}
	\space M(W_i)=M(W_j).	
	\end{align*}
	
	We defined $k$-anagram-free colouring as a generalisation of anagram-free colouring. Recall that $\phi_k(G)$ and $\phi'_k(G)$ are the $k$-anagram-free chromatic number and index respectively. In this section we show that $\phi_k$ is unbounded on graphs of maximum degree $k + 1$ and that $\phi_4$ and $\phi'_4$ are bounded on trees. The first result is a generalisation of Theorem \ref{thm:noDegreeBound}, that $\phi$ is unbounded on graphs of maximum degree $3$. The second result contrasts with Theorems \ref{thm:noDegreeBoundEdge} and \ref{thm:treeDegreeBound} because it shows a shift in behaviour from unbounded to bounded as $k$ increases.
	
	\subsection{Lower bounds}
	We now prove Theorem \ref{thm:noKPowerBound}, which says that $\phi_k$ is not bounded by a function of maximum degree. The method is similar to that used in Theorem \ref{thm:noDegreeBound} to prove $\phi$ is not bounded by maximum degree. For each $k$ and $c$ we recursively construct a graph $G$ such that every $c$-colouring of $G$ contains a $k$-anagram. The proof generalises the $k = 2$ case in the sense that Theorem \ref{thm:noDegreeBound} is implied by Theorem \ref{thm:noKPowerBound}.
	
	\begingroup
	\def\thetheorem{\ref{thm:noKPowerBound}}
	\begin{theorem}
		For $k\geq 2$, the $k$-anagram-free chromatic number is unbounded on graphs of maximum degree $k+1$.
	\end{theorem}
	\addtocounter{theorem}{-1}
	\endgroup
	\begin{proof}
		Let $S(t)$ be the statement that there exists a graph $G$ with $\Delta(G)\leq k+1$ and special vertices $u$ and $v$ with $\deg(u)=\deg(v)=1$ such that for every vertex colouring of $G$ at least one of the following holds:
		\begin{itemize}
			\item $G$ contains a $k$-anagram, or
			\item $|D| \geq \left(\frac{k}{k-1}\right)^t$, where $D$ is the set of colour multisets on $uv$-paths with length $4t$.
		\end{itemize}
	
		\begin{claim*}
			$S(t)$ is true for all $t \geq 1$.
		\end{claim*}
		\begin{proof}
			We proceed by induction on $t$. First we prove $S(1)$. Let $G$ be the graph obtained from a path $P$ of order $k$ by adding vertices $a$ and $b$, each adjacent to every vertex in $P$. Add vertices $u$ and $v$ to $G$ as well as edges $ua$ and $vb$. Note that $G$ satisfies the degree requirements of $S(1)$. Fix a colouring of $G$. If $P$ is monochromatic then it is a $k$-anagram. If $P$ is not monochromatic, then there are two paths $u,a,p_1,b,v$ and $u,a,p_2,b,v$ with distinct colour multisets. Therefore $|D| \geq 2 \geq \frac{k}{k-1}$ and so $S(1)$ is true. This graph is shown in Figure \ref{fig:kPowerbase} for $k = 4$.
			
			Assume $S(t)$ is true for some $t \geq 1$. Let $G_1, \ldots, G_k$ be copies of the graph guaranteed to exist by $S(t)$. Denote the two special vertices of $G_i$ by $u_i$ and $v_i$. Let $G$ be the graph with $V(G) = \{u,v,a,b\}$ and $E(G) = \{ua, vb\}$. Add $G_1, \ldots, G_k$ to $G$ as disjoint components. Finally, add the following edges to $G$:
			\begin{enumerate}[(i)]
				\item $u_ia$ for all $i \in [k]$.
				\item $v_ib$ for all $i \in [k]$.
				\item $u_iu_{i+1}$ for all even $i \in [k-1]$.
				\item $v_iv_{i+1}$ for all odd $i \in [k-1]$.
			\end{enumerate}
			This construction is shown in Figure \ref{fig:kPowerInduction} for $k = 4$.
			
			First we show that $G$ satisfies the degree requirements of $S(t+1)$. Clearly, $\deg(u)=\deg(v)=1$ and $\deg(a) =\deg(b)= k+1$. Each $u_i$ has degree $1$ in $G_i$ so $\deg(u_i) \leq 3 \leq k+1$. Similarly, $\deg(v_i) \leq 3 \leq k+1$. Every other vertex has the same degree as in $G_i$, which is at most $k+1$.
			
			Now fix a colouring of $G$. If some $G_i$ contains a $k$-anagram then $S(t+1)$ is satisfied so assume that each $G_i$ is $k$-anagram-free. Let $D_i$ be the set of colour multisets on paths of length $4t$ in $G_i$ with endpoints $u_i$ and $v_i$. By $S(t)$, we have $|D_i| \geq (k/(k-1))^t$ for all $i \in [k]$. We now split the proof into two cases.
				
			\begin{figure}
				\centering
				\begin{minipage}{.45\textwidth}
					\centering
					\begin{tikzpicture}
						[line width=1.4pt,vertex/.style={circle,inner sep=0pt,minimum size=0.15cm}, scale = 1]
						
						\node[draw = black, fill = black] (u) at     ($(-2.5,0)$) [vertex] [label=$u$]{};
						\node[draw = black, fill = black] (a) at     ($(-1.4,0)$) [vertex] [label=$a$]{};
						
						\node[draw = black, fill = black] (b) at     ($(1.4,0)$) [vertex] [label=$b$]{};
						\node[draw = black, fill = black] (v) at     ($(2.5,0)$) [vertex] [label=$v$]{};
						
						\node[draw = black, fill = black] (p1) at     ($(0,1.6)$) [vertex] {};
						\node[draw = black, fill = black] (p2) at     ($(0,0.525)$) [vertex] {};
						\node[draw = black, fill = black] (p3) at     ($(0,-0.525)$) [vertex] {};
						\node[draw = black, fill = black] (p4) at     ($(0,-1.6)$) [vertex] {};
						
						\draw[draw = black](u)--(a);
						\draw[draw = black](v)--(b);
						\draw[draw = black](p1)--(a);
						\draw[draw = black](p1)--(b);
						\draw[draw = black](p2)--(a);
						\draw[draw = black](p2)--(b);
						\draw[draw = black](p3)--(a);
						\draw[draw = black](p3)--(b);
						\draw[draw = black](p4)--(a);
						\draw[draw = black](p4)--(b);
						
						\draw[draw = black](p1)--(p2)--(p3)--(p4);
						
					\end{tikzpicture}
					\captionof{figure}
					{A graph satisfying $S(1)$ for \newline $k = 4$.}
					\label{fig:kPowerbase}
				\end{minipage}
				\hspace{10pt}
				\begin{minipage}{.45\textwidth}
					\centering
					\begin{tikzpicture}
							[line width=1.4pt,vertex/.style={circle,inner sep=0pt,minimum size=0.15cm}, scale = 1]
						
						\node[draw = black, fill = black] (u) at     ($(-3  ,0)$) [vertex] [label=$u$]{};
						\node[draw = black, fill = black] (a) at     ($(-2.2,0)$) [vertex] [label=$a$]{};
						
						\node[draw = black, fill = black] (b) at     ($(2.2,0)$) [vertex] [label=$b$]{};
						\node[draw = black, fill = black] (v) at     ($(3  ,0)$) [vertex] [label=$v$]{};
						
						\node[draw = black, fill = black, label = {[label distance=-5pt]135:$u_1$}] (u1) at     ($( -1,1.5)$) [vertex] {};
						\node[draw = black, fill = black, label = {[label distance=-5pt]135:$u_2$}] (u2) at     ($( -1,0.5)$) [vertex] {};
						\node[draw = black, fill = black, label = {[label distance=-5pt]225:$u_3$}] (u3) at     ($(-1,-0.5)$) [vertex] {};
						\node[draw = black, fill = black, label = {[label distance=-5pt]225:$u_4$}] (u4) at     ($(-1,-1.5)$) [vertex] {};
						                                                            
						\node[draw = black, fill = black, label = {[label distance=-5pt]45:$v_1$}] (v1) at     ($(  1,1.5)$) [vertex] {};
						\node[draw = black, fill = black, label = {[label distance=-5pt]45:$v_2$}] (v2) at     ($(  1,0.5)$) [vertex] {};
						\node[draw = black, fill = black, label = {[label distance=-5pt]315:$v_3$}] (v3) at     ($( 1,-0.5)$) [vertex] {};
						\node[draw = black, fill = black, label = {[label distance=-5pt]315:$v_4$}] (v4) at     ($( 1,-1.5)$) [vertex] {};
						
						\draw[draw = black](u)--(a);
						\draw[draw = black](v)--(b);
						\draw[draw = black](u1)--(a);
						\draw[draw = black](u2)--(a);
						\draw[draw = black](u3)--(a);
						\draw[draw = black](u4)--(a);
						\draw[draw = black](v1)--(b);
						\draw[draw = black](v2)--(b);
						\draw[draw = black](v3)--(b);
						\draw[draw = black](v4)--(b);
						
						\draw[draw = black](v1)--(v2);
						\draw[draw = black](u2)--(u3);
						\draw[draw = black](v3)--(v4);
						
						\draw[draw = black](v1)--(u1) node [midway, ellipse, draw = gray, fill = lightgray, minimum width = 55] {$G_1$};
						\draw[draw = black](v2)--(u2) node [midway, ellipse, draw = gray, fill = lightgray, minimum width = 55] {$G_2$};
						\draw[draw = black](v3)--(u3) node [midway, ellipse, draw = gray, fill = lightgray, minimum width = 55] {$G_3$};
						\draw[draw = black](v4)--(u4) node [midway, ellipse, draw = gray, fill = lightgray, minimum width = 55] {$G_4$};
					\end{tikzpicture}
					\captionof{figure}
					{A graph satisfying $S(t)$ for $k = 4$. $G_1$, $G_2$, $G_3$ and $G_4$ are graphs satisfying $S(t-1)$. Vertices $v_i$ and $u_i$ are the special vertices of degree $1$ in $G_i$.}
					\label{fig:kPowerInduction}
				\end{minipage}
			\end{figure}	
			
			In the first case there exists a colour multiset $M$ such that $M \in D_i$ for all $i \in [k]$. This means that each $G_i$ contains a $u_iv_i$-path $P_i$ with $M(P_i) = M$. Type (iii) and (iv) edges between special vertices of $G_i$ and $G_{i+1}$ mean that the subgraph induced by $P_1 \cup P_2 \cup \ldots \cup P_k$ is a path. This path is a $k$-anagram and so the colouring satisfies $S(t+1)$.
			
			In the second case there is no colour multiset that occurs in every $D_i$. Define the union of colour multisets as $\mathcal{U} := \bigcup_{i \in [k]} D_i$. For a colour multiset $M \in \mathcal{U}$ let $f(M)$ be the number of sets from $\{D_1, \ldots, D_k\}$ that contain $M$. No colour multiset occurs in every $D_i$ so $f(M) \leq k-1$ and therefore
			\begin{align*}
			|\mathcal{U}|(k - 1) \geq \sum_{M \in \mathcal{U}} f(M) = \sum_{i \in [k]} |D_i| \geq k \left(\frac{k}{k-1}\right)^t 
			\end{align*}
			Thus $|\mathcal{U}| \geq \left(\frac{k}{k-1}\right)^{t+1}$. There is a bijection from $\mathcal{U}$ to $D$ because every $uv$-path of length $4t + 4$ shares vertices $a$, $b$, $u$ and $v$. Therefore $|D|=|\mathcal{U}|$ and so $S(t+1)$ is satisfied.
		\end{proof}
		
		Let $c \geq 1$ be a number of colours and let $t$ be sufficiently large so that
		\begin{align*}
		\left(\frac{k}{k-1}\right)^t > (4t + 2)^c.
		\end{align*}	
		Let $G$ be the graph guaranteed to exist by $S(t)$ and fix an arbitrary $c$-colouring of $G$. Let $D$ be the set of colour multisets as defined previously. By (\ref{eqn:colourMult}) there are at most $(4t + 2)^c$ colour multisets of size $4t + 1$. Therefore $|D| \leq (4t + 2)^c < \left(\frac{k}{k-1}\right)^t$. Thus $G$ contains a $k$-anagram by $S(t)$.
	\end{proof}
	The following natural question arises: Does there exist a $d$ such that $\phi_k$ is unbounded on graphs of maximum degree $d$ for all $k \geq 3$? Also, the analogous problem for edge colouring is open. We know of no family of graphs for which $\phi'_k$ is unbounded, except in the case of $k = 2$.
	
	\subsection{Upper bounds on trees}
	In this section we prove an upper bound on $k$-anagram-free colouring that contrasts with the results for anagram-free colouring.
	
	\begingroup
	\def\thetheorem{\ref{thm:kPowerTreeBound}}
	\begin{theorem}
		If $T$ is a tree and $k \geq 4$ then $\phi_k(T) \leq 4$ and $\phi'_k(T) \leq 4$.
	\end{theorem}
	\addtocounter{theorem}{-1}
	\endgroup
	\begin{proof}
		Root $T$ at an arbitrary vertex $r\in V(T)$ and let $h$ be the height of the resulting rooted tree. Let $C = c_o\ldots c_h$ be an anagram-free word on four symbols. Colour each vertex $u \in V(T)$ by $c_x$ where $x$ is the distance between $u$ and $r$.
		
		Let $P=P_1\ldots P_k$ be a path in $T$ such that, for some $n \geq 1$, $|V(P_i)|= n$ for all $i \in [k]$. Note that $P$ is a $k$-anagram if and only if $M(P_1)= M(P_2)= \ldots = M(P_k)$. We now show that $P$ is not a $k$-anagram.
		
		$P$ contains a unique vertex $v$ closest to $r$. If $v \in V(P_i)$ with $i \geq 3$ then the colour sequence along $P_1 \cup P_2$ appears in $C$, so $M(P_1) \neq M(P_2)$. In the other case, $v \in V(P_i)$ with $i \leq 2$. Then the colour sequence along $P_3 \cup P_4$ appears in $C$, so $M(P_3) \neq M(P_4)$. In each case, $P$ is not a $k$-anagram. Hence $\phi_k(T) \leq 4$.
		
		The proof can be repeated for $\phi'_k$ by modifying the definition of $P$. In this case $P$ has length $nk$ and its subpaths $P_1, P_2, \ldots, P_k$ are edge-disjoint such that $P$ is a $k$-anagram if and only if $M(P_1)= M(P_2)= \ldots = M(P_k)$.
	\end{proof}
	Theorem \ref{thm:kPowerTreeBound} demonstrates a qualitative change in behaviour as $k$ increases. The case of $k = 3$ is an open problem that sits between bounded and unbounded behaviour. For $\phi_3$ on trees we have upper bounds due to Equation (\ref{eqn:kChain}) and Theorems \ref{thm:treeHeightBound} and \ref{thm:treePathwidthBound}. We prove similar upper bounds for $\phi'_3$ using Dekking's \cite{dekking1979strongly} result $\phi_3(P) = 3$.
	\begingroup
	\def\thetheorem{\ref{thm:kTreeIndexPathwidth}}
	\begin{theorem}
		For every tree $T$, $\phi'_3(T) \leq 4\pw(T)$.
	\end{theorem}
	\addtocounter{theorem}{-1}
	\endgroup
	\begin{proof}
		The proof is by induction on $m$. The base case is satisfied because trees of pathwidth $0$ are edgeless. Now assume that every tree $T$ with pathwidth at most $m$ has $\phi'_3(T) \leq 4m$.
		
		Let $T$ be a tree of pathwidth $m+1$. By Lemma \ref{lem:mainPath} there exists a path $P \subseteq T$ such that $\pw(T - V(P)) \leq m$. Each component of $T - V(P)$ has pathwidth at most $m$, so can be $3$-anagram-free edge-coloured with the same set of $4m$ colours, by induction. We now use four additional colours to colour the remaining edges. Dekking \cite{dekking1979strongly} proves $\phi_3(P) = 3$ so we can $3$-anagram-free edge-colour $P$ with three colours. The fourth extra colour is used to colour the edges between $P$ and $T - V(P)$.
	
		We now show that this colouring is $3$-anagram-free. Let $Q$ be a path in $T$. If $Q$ is entirely contained within a component of $T - V(P)$ then, by induction, $Q$ is not a $3$-anagram. Otherwise $Q$ intersects $P$. The intersection of $Q$ and $P$ is a $3$-anagram-free subpath of $P$ and the colours in $Q \cap P$ occur nowhere else in $Q$. Therefore $Q$ is not a $3$-anagram.
	\end{proof}
	Note that no similar bound exists for $\phi'_2$ because stars have pathwidth $1$ and $\phi'_2$ is unbounded on stars. A bound on $\phi'_3(T)$ as a function of radius follows from the relation between pathwidth and radius in trees. Note that since $\phi_3(P) \leq 3$ we are also able to prove that $\phi_3(T) \leq 3\pw(T) + 1$ with a proof similar to the proof of Theorem \ref{thm:colourPathwidth}.
	
	Dekking also proves $\phi_4(P) = 2$ and we use both results to improve upon Theorem \ref{thm:kPowerTreeBound} for larger $k$.  
	\begin{theorem}\label{thm:kPowerTreeBoundVertGeneral}
		For all $z \geq 1$ and $k \geq 2z$, if  $\phi_z(P) \leq y$ for all paths $P$ then $\phi_k(T) \leq y$ for all trees $T$.
	\end{theorem}
	\begin{proof}
		Let $T$ be a tree with root $r$ and height $h$. Let $C = c_o\ldots c_h$ be a $z$-anagram-free word on $y$ symbols. Colour each vertex $u \in V(T)$ by $c_x$ where $x$ is the distance between $u$ and $r$.
		
		Let $P=P_1\ldots P_k$ be a path in $T$ such that $|V(P_i)|= n$ for some $n \geq 1$. Note that $P$ is a $k$-anagram if and only if $M(P_1)= M(P_2)= \ldots = M(P_k)$. We now show that $P$ is not a $k$-anagram.
		
		$P$ contains a unique vertex $v$ closest to $r$. If $v \in V(P_i)$ with $i > z$ then the colour sequence along $P_1 \cup P_2\cup \ldots \cup P_z$ appears in $C$. In the other case, $v \in V(P_i)$ with $i \leq z$, which implies the colour sequence along $P_{z+1} \cup P_{z+2}\cup \ldots \cup P_{2z}$ appears in $C$. In each case there exist $a,b$ such that $M(P_a) \neq M(P_b)$ because $C$ is $z$-anagram-free. Therefore $P$ is not a $k$-anagram.
	\end{proof}
	Dekking \cite{dekking1979strongly} proves $\phi_3(P) = 3$ and $\phi_4(P) = 2$, so Theorem \ref{thm:kPowerTreeBoundVertGeneral} implies $\phi_6(T) \leq 3$ and $\phi_8(T) \leq 2$. The corresponding results for edge colouring, $\phi'_6(T) \leq 3$ and $\phi'_8(T) \leq 2$, are achieved with modifications similar to those at the end of Theorem \ref{thm:kPowerTreeBound}.
	
	\section{Open Problems}
	Throughout this paper we have posed many conjectures and open problems. In this section we provide a summary as well as some further questions. 
	
	The results of Sections \ref{sec:lower} and \ref{sec:upperTree} motivate further study of $\phi$ on trees. Whether $\phi$ is bounded on the complete binary tree is a particularly interesting question. We also conjecture the result analogous to \cite{currie2002there}, that $\phi(C) \leq 4$ for cycles with only finitely many exceptions. The tight bounds given in Theorems \ref{thm:treeHeightBound} and \ref{thm:treePathwidthBound} motivate further investigation of pathwidth and radius. We ask whether pathwidth is tied to $\phi$ on trees, that is, whether there exists a function$f$ such that $\pw(T) \leq f(\phi(T))$ for every tree $T$. Pathwidth is unbounded on complete binary trees so the two questions are related.
	
	Section \ref{sec:kPower} contains two open problems. The first is whether $\phi'_k$ is bounded by maximum degree for some $k \geq 3$. The second is whether $\phi_3$ and $\phi'_3$ are bounded on trees.
	
	\subsection*{Acknowledgements} Thanks to Gwena\"el Joret for
	stimulating discussions.
	
	\subsection*{Note} At the same time as the present paper was
	completed, Kam{\v{c}}ev, {\L}uczak, and Sudakov \cite{KLS16} posted a paper on the arXiv that independently introduces \emph{anagram-free} graph colouring. Each paper independently proves some of the results in the other paper. Note that Kam{\v{c}}ev, {\L}uczak, and Sudakov proved that $\phi$ is unbounded on complete binary trees, thus solving the above-mentioned open problem.
	

\begin{thebibliography}{22}
	
	\bibitem{alon2002nonrepetitive}
	Noga Alon, Jaros{\l}aw Grytczuk, Mariusz Ha{\l}uszczak, and Oliver Riordan.
	\newblock Nonrepetitive colorings of graphs.
	\newblock {\em Random Structures \& Algorithms}, 21(3-4):336--346, 2002.
	
	\bibitem{brevsar2007nonrepetitivetree}
	Bo{\v{s}}tjan Bre{\v{s}}ar, Jaros{\l}aw Grytczuk, Sandi Klav{\v{z}}ar, Staszek
	  Niwczyk, and Iztok Peterin.
	\newblock Nonrepetitive colorings of trees.
	\newblock {\em Discrete Mathematics}, 307(2):163--172, 2007.
	
	\bibitem{britten1971repetitive}
	Roy~J Britten and Eric~H Davidson.
	\newblock Repetitive and non-repetitive {DNA} sequences and a speculation on
	  the origins of evolutionary novelty.
	\newblock {\em Quarterly Review of Biology}, 46(2):111--138, 1971.
	
	\bibitem{cori1990partially}
	Robert Cori and Maria~Rosaria Formisano.
	\newblock Partially abelian squarefree words.
	\newblock {\em Informatique th{\'e}orique et applications}, 24(6):509--520,
	  1990.
	
	\bibitem{cummings1996strongly}
	Larry~J Cummings.
	\newblock Strongly square-free strings on three letters.
	\newblock {\em Australasian J. Combinatorics}, 14:259--266, 1996.
	
	\bibitem{currie2002there}
	James~D Currie.
	\newblock There are ternary circular square-free words of length $n$ for $n
	  \geq 18$.
	\newblock {\em Electron. J. Combin}, 9(1):N10, 2002.
	
	\bibitem{dekking1979strongly}
	Frederik~Michel Dekking.
	\newblock Strongly non-repetitive sequences and progression-free sets.
	\newblock {\em J. Combinatorial Theory, Series A}, 27(2):181--185, 1979.
	
	\bibitem{dujmovic2011nonrepetitive}
	Vida Dujmovi{\'c}, Gwena{\"e}l Joret, Jakub Kozik, and David~R Wood.
	\newblock Nonrepetitive colouring via entropy compression.
	\newblock \emph{Combinatorica}, pages 1--26, 2015.
	
	\bibitem{fertin2004star}
	Guillaume Fertin, Andr{\'e} Raspaud, and Bruce Reed.
	\newblock Star coloring of graphs.
	\newblock {\em J. Graph Theory}, 47(3):163--182, 2004.
	
	\bibitem{gkagol2016pathwidth}
	Adam Gagol, Gwena{\"e}l Joret, Jakub Kozik, and Piotr Micek.
	\newblock Pathwidth and nonrepetitive list coloring.
	\newblock {\em Electron. J. Combin.}, 23(4):\#P4.40, 2016
	
	\bibitem{grytczuk2007nonrepetitive}
	Jaros{\l}aw Grytczuk.
	\newblock Nonrepetitive colorings of graphs a survey.
	\newblock {\em Int. J. Math. Math. Sci.}, 2007:Art. ID 74639.
	
	\bibitem{grytczuk2006nonrepetitive}
	Jaros{\l}aw Grytczuk.
	\newblock Nonrepetitive graph coloring.
	\newblock In {\em Graph Theory in Paris}, Trends in Mathematics,
	  209--218. Birkhauser, 2007.
	
	\bibitem{grytczuk2013new}
	Jaros{\l}aw Grytczuk, Jakub Kozik, and Piotr Micek.
	\newblock New approach to nonrepetitive sequences.
	\newblock {\em Random Structures \& Algorithms}, 42(2):214--225, 2013.
	
	\bibitem{harant2012nonrepetitive}
	Jochen Harant and Stanislav Jendro{\soft{l}}.
	\newblock Nonrepetitive vertex colorings of graphs.
	\newblock {\em Discrete Math.}, 312(2):374--380, 2012.
	
	\bibitem{KLS16}
	Nina Kam{\v{c}}ev, Tomasz {\L}uczak, and Benny Sudakov.
	\newblock Anagram-free colorings of graphs.
	\newblock 2016.
	\newblock arXiv:1606.09062.
	
	\bibitem{keranen1992abelian}
	Veikko Ker{\"a}nen.
	\newblock Abelian squares are avoidable on $4$ letters.
	\newblock In {\em Automata, languages and programming}, volume 623 of {\em
	  Lecture Notes in Comput. Sci.}, pages 41--52. Springer, 1992.
	
	\bibitem{keranen2009abelian}
	Veikko Ker{\"a}nen.
	\newblock A powerful abelian square-free substitution over 4 letters.
	\newblock {\em Theoret. Comput. Sci.}, 410(38-40):3893--3900, 2009.
	
	\bibitem{book:sparsity}
	Jaros{\l}av Ne{\v{s}}et{\v{r}}il and Patrice~Ossona De~Mendez.
	\newblock {\em Sparsity: Graphs, Structures, and Algorithms}.
	\newblock Springer, 2012.
	
	\bibitem{richmond2009counting}
	L.~B. Richmond and Jeffrey Shallit.
	\newblock Counting abelian squares.
	\newblock {\em Electron. J. Combin.}, 16(1):\#R72, 2009.
	
	\bibitem{stanley2011enumerative}
	Richard~P Stanley.
	\newblock {\em Enumerative Combinatorics}, volume~1.
	\newblock Cambridge University Press, Cambridge, second edition, 2011.
	
	\bibitem{suderman2004pathwidth}
	Matthew Suderman.
	\newblock Pathwidth and layered drawings of trees.
	\newblock {\em Internat. J. Comput. Geom. Appl.}, 14(3):203--225, 2004.
	
	\bibitem{thue1914probleme}
	Axel Thue.
	\newblock Probleme {\"u}ber ver{\"a}nderungen von zeichenreihen nach gegebenen
	  regeln.
	\newblock pages I. Math. naturv. Klasse, 10. Christiana Videnskabs-Selskabs
	  Skrifte, 1914.
	
	\end{thebibliography}
	
	\def\soft#1{\leavevmode\setbox0=\hbox{h}\dimen7=\ht0\advance \dimen7
	  by-1ex\relax\if t#1\relax\rlap{\raise.6\dimen7
	  \hbox{\kern.3ex\char'47}}#1\relax\else\if T#1\relax
	  \rlap{\raise.5\dimen7\hbox{\kern1.3ex\char'47}}#1\relax \else\if
	  d#1\relax\rlap{\raise.5\dimen7\hbox{\kern.9ex \char'47}}#1\relax\else\if
	  D#1\relax\rlap{\raise.5\dimen7 \hbox{\kern1.4ex\char'47}}#1\relax\else\if
	  l#1\relax \rlap{\raise.5\dimen7\hbox{\kern.4ex\char'47}}#1\relax \else\if
	  L#1\relax\rlap{\raise.5\dimen7\hbox{\kern.7ex
	  \char'47}}#1\relax\else\message{accent \string\soft \space #1 not
	  defined!}#1\relax\fi\fi\fi\fi\fi\fi}

\end{document}